\documentclass[12pt,a4paper,twoside]{article}

\usepackage[a4paper,
            left=2.5cm,
            right=2.5cm,
            top=3cm,
            bottom=3cm]{geometry}

\usepackage[T1]{fontenc}
\usepackage[utf8]{inputenc}
\usepackage{libertinus}
\usepackage{microtype}
\usepackage[english]{babel}

\usepackage{amsmath,amssymb,amsthm,amsfonts}
\usepackage{mathrsfs}
\usepackage{mathtools}
\usepackage{stmaryrd}
\usepackage{enumitem}
\usepackage{titlesec}
\usepackage{titletoc}
\usepackage{fancyhdr}
\usepackage{hyperref}
\usepackage{xcolor}
\usepackage{graphicx}
\usepackage{setspace}
\usepackage{tocloft}
\usepackage{etoolbox}
\usepackage{tikz}
\usepackage{tikz-cd}
\usetikzlibrary{arrows}
\usetikzlibrary{calc}
\usepackage{color}
\usepackage{here}
\usepackage{mathrsfs}
\usepackage{pgf,tikz}
\usetikzlibrary{decorations.pathreplacing, shapes.multipart, arrows, matrix, shapes}
\usetikzlibrary{patterns}
\usepackage{caption}
\usepackage{enumitem}
\usepackage{cancel}
\usepackage{comment}
\usepackage{verbatim}
\usepackage{indentfirst}

\hypersetup{
    colorlinks=true,
    linkcolor=black,
    citecolor=black,
    urlcolor=black,
    pdftitle={},
    pdfauthor={}
}

\titleformat{\section}
  {\normalfont\bfseries}
  {\thesection.}
  {0.5em}
  {}

\titleformat{\subsection}
  {\normalfont\bfseries}
  {\thesubsection.}
  {0.5em}
  {}

\titleformat{\subsubsection}
  {\normalfont\itshape}
  {\thesubsubsection.}
  {0.5em}
  {}

\titlespacing*{\section}
{0pt}{18pt}{8pt}

\titlespacing*{\subsection}
{0pt}{14pt}{6pt}

\titlespacing*{\subsubsection}
{0pt}{10pt}{4pt}

\theoremstyle{plain}
\newtheorem{theorem}{Theorem}[section]
\newtheorem{maintheorem}{Theorem}
\newtheorem{proposition}[theorem]{Proposition}
\newtheorem{lemma}[theorem]{Lemma}
\newtheorem{corollary}[theorem]{Corollary}

\theoremstyle{definition}
\newtheorem{definition}[theorem]{Definition}

\newtheorem{remark}[theorem]{Remark}

\newcommand{\supp}{\operatorname{supp}}

\newcommand{\Z}{{\mathbb Z}}
\newcommand{\R}{{\mathbb R}}

\newcommand{\T}{{\mathbb{S}^1}}

\newcommand{\N}{{\mathbb N}}

\newcommand{\Tf}{{\mathbb{S}_f^1}}

\newcommand{\cA}{{\mathcal A}}

\newcommand{\cE}{{\mathcal E}}

\newcommand{\cP}{{\mathcal P}}

\newcommand{\eps}{\varepsilon}
\newcommand{\per}{\operatorname{Per}(f)}

\newcommand{\eqdef}{\stackrel{\scriptscriptstyle\textrm def.}{=}}

\newcommand{\Leb}{\operatorname{Leb}}
\newcommand{\Lip}{\textrm{Lip}}

\newcommand{\tix}{\tilde{x}}

\newcommand{\tg}{\tilde{\Gamma}}
\newcommand{\hf}{\hat{f}}

\newcommand{\cps}{\cP_f^{\operatorname{erg}}(\mathbb{S}^1)}
\newcommand{\emp}{\mathsf e}

\newcommand{\Wone}{W_1}
\newcommand{\diam}{\operatorname{diam}}

\newcommand{\Sone}{\mathbb{S}^1}
\newcommand{\wto}{\overset{*}{\rightharpoonup}}

\makeatletter

\makeatletter

\renewcommand{\maketitle}{

\thispagestyle{empty}

\begin{center}

{\Large\bfseries
Low emergence in one-dimensional dynamics
\par}

\vspace{1cm}

{%
\large
Odylo Costa\footnote{%
Departamento de Matemática, Universidade Federal do Ceará,
Fortaleza, Brazil. \texttt{odylo.costa@mat.ufc.br}}

Bruno Santiago\footnote{%
Instituto de Matemática, Universidade Federal Fluminense,
Niterói, Brazil. \texttt{brunosantiago@id.uff.br}}
\par
}

\end{center}

\vspace{1cm}

\noindent\textbf{Abstract.} ---
We show that every $C^2$ immersion of the circle $\mathbb{S}^1$ into itself has low emergence.      
\vspace{0.4cm}

\tableofcontents

\vspace{0.6cm}

\noindent\textbf{Mathematical subject classification (2020).} --- 37E10, 37C40, 37A35.

\smallskip

\noindent\textbf{Keywords.} --- Expanding maps, emergence.

\smallskip

\noindent\textbf{Acknowledgements.} --- B.S. was supported by \emph{Conselho Nacional de Densenvolvimento Científico e Tecnológico (CNPQ)} via the grant \emph{Bolsa PQ 307994/2025-2} and by \emph{Fundação de Amparo à Pesquisa do Estado do Rio de Janeiro (FAPERJ)} via the grants \emph{JCNE E-26/204.571/2024} and \emph{JPF E-26/210.344/2022}. O.C. was supported by Instituto Serrapilheira grant “Jangada Dinâmica: Impulsionando Sistemas Dinâmicos na Região Nordeste”, and he is thankful to Jamerson Bezerra (UFC) for useful conversations around some topics of this paper. B.S. and O.C. were partially supported by \emph{Coordenação de Aperfeiçoamento de Pessoal de Nível Superior (CAPES) - Finance code 001}.

\vspace{0.5cm}

\hrule

\vspace{0.3cm}

\newpage

}

\makeatother

\begin{document}

\title{Low emergence in one-dimensional dynamics}
\author{Odylo Costa \& Bruno Santiago}
\date{\today}

\maketitle

\section{Introduction}

Consider a discrete-time dynamical system acting on some compact space. The primary goal of ergodic theory is the description of the asymptotic statistical behaviors of initial conditions. From a global perspective, such a task seems to be simply too general to be within reach. 

The concept of emergence coined by Berger \cite{Berger2017} enters into the theory as a conceptual framework that may fill this gap. Indeed, given a large class of dynamical systems, one can separate its elements into two very distinct subclasses: those with low emergence and those with high emergence. Systems with high emergence present no hope for a simple coherent statistical analysis. In this paper we aim at describing the emergence of one-dimensional dynamical systems.

It is worth pointing out that the emergence of a system depends on the choice of a particular reference measure. The situation may change depending on whether the reference measure is invariant or not. In particular, Berger and Bochi \cite{BergerBochiAIM2021} show that the \emph{metric emergence} with respect to any invariant measure is bounded above by the \emph{topological emergence}.

From this perspective, the article \cite{CarvalhoRodriguesVarandas2024Emergence} shows that orientation-preserving homeomorphisms of one-dimensional manifolds have low \emph{topological emergence} and, in particular, have low \emph{metric emergence} with respect to any invariant reference measure. This is an indication that in dimension one one may hope for a tame global behavior.

Nevertheless, the work of Hofbauer-Keller \cite{HofbauerKeller1995MaxOsc} on the quadratic family implies that high emergence with respect to a non-invariant reference measure may appear if we allow critical points.

Our goal in this paper is to complement these works by considering maps of the unit circle, not necessarily invertible, but without critical points. We take as natural reference measure the (not necessarily invariant) Haar measure of the circle. Our main result is the following.

\begin{maintheorem}
    \label{main}
Let $f:\mathbb{S}^1\to\mathbb{S}^1$ be a $C^2$ immersion. Then, $f$ has low emergence.
\end{maintheorem}

We point out that it is known that a particular immersion of the circle may present complicated non-coherent behavior, such as in \cite{Zweimuller2002ExactCovering}. Thus, our result complements the existing gap in the literature by showing that although non-coherent behavior may occur, the behavior in general is not highly complex.

Our proof combines the powerful analysis far from the critical set done by Mañé \cite{mane} with a general entropy formula for asymptotic measures \cite{CYZ}. Let us briefly sketch the argument. Using several results from \cite{mane}, one can reduce the study to a particular case of immersions with finitely many indifferent fixed points, topologically conjugate to expanding maps and with no physical measures. Notice that a topological conjugacy does not preserve low emergence; thus further arguments are needed. Therefore, using the main technical result from \cite{mane}, we show that every asymptotic measure not supported on the finite set of indifferent fixed points must have a positive Lyapunov exponent. The entropy formula \cite{CYZ} then shows the existence of physical measures, leading to a contradiction.   

This paper is organized as follows: in Section 2, we give the necessary definitions and basic results about emergence. In Section 3, we treat the case of homeomorphisms of one-dimensional manifolds and establish low metric emergence in this case. Notice that the results of this section generalize \cite{CarvalhoRodriguesVarandas2024Emergence}. In Section 4, we explain how to use the partially hyperbolic theory of \cite{CYZ} in order to deduce an entropy formula for asymptotic measures. In Section 5, we recall the results of \cite{mane} and then in Sections 6 and 7 we reduce the problem to the particular case alluded to above. In Section 8 we show that asymptotic measures not supported on the indifferent set have positive exponent and conclude the proof.

\section{Empirical measures and emergence}
	
	We recall the definition of emergence for a
	continuous map with respect to an arbitrary reference measure.
	
	\begin{definition}[Empirical measures]
		Let $(X,d)$ be a compact metric space, $f\colon X\to X$ continuous,
		and $x\in X$.
		The \textit{$n$-th empirical measure associated to} $x$ is
		\[
		\emp_n(x) := \frac1n \sum_{k=0}^{n-1} \delta_{f^k(x)}
		\in \cP(X),
		\]
		where $\cP(X)$ denotes the space of Borel probability measures on $X$. 
	\end{definition}

When necessary, we use the superscript $\emp_n^f(x)$ to mark the dependence on the dynamical system. An important set in this paper is the set of possible accumulation points of a sequence $(\emp_n(x))_n$, for $x\in X$.

    \begin{definition}[Asymptotic measures]
	Let $(X,d)$ be a compact metric space, $f\colon X\to X$ continuous,
	and $x\in X$. We denote by
	\[
	\cA_f(x)
	:=
	\left\{
		\nu\in\cP(X):
		\emp_{n_k}(x)\wto\nu
		\text{ for some }n_k\to\infty
		\right\}
	\]
	the set of \emph{asymptotic measures} of \(x\).
\end{definition}

\begin{remark}\label{rem:asymptotic-invariant}
	By the usual Krylov--Bogolyubov argument, every asymptotic measure is \(f\)-invariant. More precisely, for every \(x\in X\),
	\[
	\cA_f(x)\subset\cP_f(X),
	\]
	where \(\cP_f(X)\) denotes the set of \(f\)-invariant Borel probability measures on \(X\). See, for instance, \cite[Lemma~2.2.4]{viana2016foundations}.
\end{remark}

    Let $\Wone$ be Kantorovich--Wasserstein--1 distance on $\cP(X)$. This metric is compatible with the weak-$\ast$ topology and is defined by:
	$$
	\Wone(\mu,\nu)=\inf_{\pi \in \Pi(\mu,\nu)}\int_{X\times X} d(x,y) \ d\pi(x,y),
	$$
	where $d$ is the metric on $X$, $\Pi(\mu,\nu)\coloneqq \left\{\pi \in \mathcal{P}(X\times X) \mid (p_{1})_{\ast}\pi=\mu \text{ and }(p_{2})_{\ast}\pi=\nu\right\}$, and $p_{i}\colon X\times X \to X$ are the projections onto the first and second coordinates.

	\begin{definition}[Emergence with respect to a reference measure]
		Let $\mu$ be a Borel probability on $X$ (not necessarily $f$-invariant).
		For $\varepsilon>0$, the \emph{emergence} of $f$ with
		respect to $\mu$ is the function $\cE_\mu(f)\colon \R_{>0}\to \N$ defined as
		\[
		\cE_\mu(f)(\varepsilon)
		:= \min\left\{
		N\in\N :
		\begin{array}{l}
			\exists \nu_1,\dots,\nu_N\in\cP(X) \text{ such that} \\[0.2em]
			\displaystyle
			\limsup_{n\to\infty}
			\int_X
			\min_{1\le i\le N} \Wone \big(\emp_n(x),\nu_i\big)\, d\mu(x)
			\le \varepsilon
		\end{array}
		\right\}.
		\]
		The \emph{order of emergence} is
		\[
		\overline{\mathcal{OE}}_\mu(f)
		:= \limsup_{\varepsilon\to0^+}
		\frac{\log\log\cE_\mu(f)(\varepsilon)}{-\log\varepsilon}.
		\]
		We say that $f$ has \emph{low emergence} (with respect to $\mu$)
		if $\overline{\mathcal{OE}}_\mu(f)=0$.
	\end{definition}

	\subsection{A general criterion for low emergence}
	
	The following elementary criterion is one of the main tools to establish low emergence. It is the abstract mechanism behind the ``finite simplex implies low emergence'' principle.
	
	\begin{lemma}\label{lem:compact-family}
	Let $(Y,d)$ be a compact metric space, let $K\subset Y$ be compact,
	and let $\{\eta_1,\dots,\eta_N\}\subset K$ be an $\eps$-net of $K$.
	If $(y_n)_{n\ge1}\subset Y$ is a sequence such that every accumulation
	point of $(y_n)$ belongs to $K$, then
	\[
	\limsup_{n\to\infty}
	\min_{1\le i\le N} d(y_n,\eta_i)
	\le \eps.
	\]
\end{lemma}

\begin{proof}
	Assume the contrary. Then there exist $\eps_0>\eps$ and a subsequence $n_k\to\infty$ such that
	\[
	\min_{1\le i\le N}d(y_{n_k},\eta_i)\ge \eps_0
	\qquad\forall k.
	\]
	Since $Y$ is compact, after passing to a further subsequence we may
	assume that $y_{n_k}\to y$ for some $y\in Y$. Then $y$ is an accumulation point of $(y_n)$, and therefore, by hypothesis, $y\in K$.
	
    Because $\{\eta_1,\dots,\eta_N\}$ is an $\eps$-net of $K$, there exists $i_0$ such that $d(y,\eta_{i_0})\le\eps$. Hence,
	\[
	\limsup_{k\to\infty}
	\min_{1\le i\le N}d(y_{n_k},\eta_i)
	\le
	\lim_{k\to\infty}d(y_{n_k},\eta_{i_0})
	=
	d(y,\eta_{i_0})
	\le\eps,
	\]
	contradicting $\min_{1\le i\le N}d(y_{n_k},\eta_i)\ge\eps_0>\eps$.
\end{proof}
	
	\begin{proposition}\label{prop:compact-family-bound}
		Let $f\colon X\to X$ be continuous on a compact metric space, let $\mu$ be a Borel probability,
		and let $K\subset \cP(X)$ be compact.
		Assume that for $\mu$-a.e. $x$, every accumulation point of $(\emp_n(x))_{n\ge1}$ belongs to $K$.
		Then for every $\eps>0$,
		\[
		\mathcal{E}_\mu(f)(\eps)\le N_{\Wone}(K,\eps),
		\]
		where $N_{\Wone}(K,\eps)$ denotes the $\eps$-covering number of $K$ for $\Wone$.
		In particular, if $K$ has polynomial covering numbers, then $f$ has low emergence with respect
		to $\mu$.
	\end{proposition}
	
	\begin{proof}
		Let $\{\eta_1,\dots,\eta_N\}\subset K$ be an $\eps$-net of $K$, with
		$N=N_{\Wone}(K,\eps)$. Define
		\[
		\Delta_n(x):=\min_{1\le i\le N}\Wone(\emp_n(x),\eta_i).
		\]
		By Lemma~\ref{lem:compact-family}, for $\mu$-a.e. $x$ one has
		\[
		\limsup_{n\to\infty}\Delta_n(x)\le \eps.
		\]
		Also $0\le \Delta_n(x)\le \diam(X)$ for all $x,n$.
		Applying Fatou to $\diam(X)-\Delta_n$ gives
		\[
		\limsup_{n\to\infty}\int_X\Delta_n(x)\,d\mu(x)
		\le
		\int_X \limsup_{n\to\infty}\Delta_n(x)\,d\mu(x)
		\le \eps.
		\]
		Thus we can take $\{\eta_1,\dots,\eta_N\}$ in the definition of emergence,
		and conclude $\mathcal{E}_\mu(f)(\eps)\le N$.
	\end{proof}

\section{The degree-one case: circle and interval homeomorphisms}

In this section we prove the low-emergence estimates needed for the
degree-one case. We also record interval versions, which will be used later
when treating periodic plateaux for immersions of $\T$.

Throughout this section, \(\mathbb S^1=\mathbb R/\mathbb Z\) is endowed with
the standard geodesic distance
\[
d_{\mathbb S^1}(\pi(s),\pi(t))
=
\min_{k\in\mathbb Z}|s-t-k|.
\]
We write
\[
D:=\diam(\mathbb S^1).
\]
Thus, in this normalization, \(D=1/2\).

\subsection{A Ces\`aro lemma}

We shall use the following elementary fact several times.

\begin{lemma}[Ces\`aro convergence]\label{lem:cesaro-homeo}
Let \((a_j)_{j\ge0}\) be a sequence of real numbers converging to \(L\).
Let \((m_n)_{n\ge1}\) be a sequence of positive integers with
\(m_n\to+\infty\). Then
\[
\frac1{m_n}\sum_{j=0}^{m_n-1}a_j\longrightarrow L.
\]
\end{lemma}

\begin{proof}
Fix \(\varepsilon>0\). Since \(a_j\to L\), there exists \(J\ge1\) such that
\[
|a_j-L|<\varepsilon
\qquad\text{for every }j\ge J.
\]
Since \((a_j)\) is convergent, it is bounded. Hence there exists \(C>0\)
such that
\[
|a_j-L|\le C
\qquad\text{for every }j\ge0.
\]
Then
\[
\left|
\frac1{m_n}\sum_{j=0}^{m_n-1}a_j-L
\right|
\le
\frac1{m_n}\sum_{j=0}^{J-1}|a_j-L|
+
\frac1{m_n}\sum_{j=J}^{m_n-1}|a_j-L|.
\]
Therefore
\[
\left|
\frac1{m_n}\sum_{j=0}^{m_n-1}a_j-L
\right|
\le
\frac{JC}{m_n}
+
\varepsilon.
\]
Letting \(n\to\infty\) and then \(\varepsilon\to0\) proves the claim.
\end{proof}

\subsection{Interval homeomorphisms}

Let \(I=[0,\ell]\) be a compact interval endowed with the usual distance
\(|x-y|\).

\begin{proposition}[Orientation-preserving interval homeomorphisms]
\label{prop:interval-preserving-homeo}
Let \(f:I\to I\) be an orientation-preserving homeomorphism. Then for every
\(x\in I\) there exists \(p_x\in\operatorname{Fix}(f)\) such that
\[
f^n(x)\longrightarrow p_x.
\]
Consequently,
\[
\emp_n(x)\wto \delta_{p_x}.
\]
Moreover, for every Borel probability \(\mu\) on \(I\), and for every
\(\varepsilon>0\),
\[
\cE_\mu(f)(\varepsilon)
\le
1+\left\lceil\frac{2\ell}{\varepsilon}\right\rceil .
\]
In particular, \(f\) has low emergence.
\end{proposition}

\begin{proof}
Since \(f\) is increasing and onto, it fixes the endpoints of \(I\). Thus
\(\operatorname{Fix}(f)\neq\emptyset\). The set \(\operatorname{Fix}(f)\)
is closed, and every connected component of
\[
I\setminus \operatorname{Fix}(f)
\]
is an open interval \((a,b)\) whose endpoints are fixed by \(f\).

On such a component, the continuous function \(x\mapsto f(x)-x\) has no
zero, hence has constant sign. If \(f(x)>x\) on \((a,b)\), then
\((f^n(x))_{n\ge0}\) is increasing and bounded above by \(b\). Thus it
converges to some \(L\in[a,b]\). By continuity, \(f(L)=L\), hence
\(L\in\{a,b\}\). Since the sequence is increasing and starts in \((a,b)\),
we must have \(L=b\). Similarly, if \(f(x)<x\) on \((a,b)\), then
\(f^n(x)\to a\). If \(x\in\operatorname{Fix}(f)\), then \(f^n(x)=x\) for
all \(n\). This proves the first claim.

Let \(\varphi\in C^0(I)\). Since \(f^n(x)\to p_x\), we have
\[
\varphi(f^n(x))\to \varphi(p_x).
\]
By Lemma~\ref{lem:cesaro-homeo},
\[
\frac1n\sum_{k=0}^{n-1}\varphi(f^k(x))
\longrightarrow
\varphi(p_x).
\]
Thus \(\emp_n(x)\wto\delta_{p_x}\).

It remains to estimate emergence. Let
\[
M:=\left\lceil\frac{2\ell}{\varepsilon}\right\rceil
\]
and choose the grid
\[
y_j:=\frac{j\ell}{M},
\qquad j=0,\dots,M.
\]
For every \(p\in I\), there exists \(j\in\{0,\dots,M\}\) such that
\[
|p-y_j|\le \frac{\ell}{2M}\le\frac{\varepsilon}{4}.
\]
Let
\[
\nu_j:=\delta_{y_j},
\qquad j=0,\dots,M.
\]
For each \(x\), choose \(j(x)\) such that
\[
W_1(\delta_{p_x},\nu_{j(x)})\le \frac{\varepsilon}{4}.
\]
Then
\[
\min_{0\le j\le M}W_1(\emp_n(x),\nu_j)
\le
W_1(\emp_n(x),\delta_{p_x})
+
W_1(\delta_{p_x},\nu_{j(x)}).
\]
Taking \(\limsup\) in \(n\), we obtain
\[
\limsup_{n\to\infty}
\min_{0\le j\le M}W_1(\emp_n(x),\nu_j)
\le
\frac{\varepsilon}{4}
\]
for every \(x\in I\). As usual, applying Fatou to
\(\ell-\min_j W_1(\emp_n(x),\nu_j)\), or equivalently using dominated
convergence for the decreasing envelopes, gives
\[
\limsup_{n\to\infty}
\int_I
\min_{0\le j\le M}W_1(\emp_n(x),\nu_j)\,d\mu(x)
\le
\frac{\varepsilon}{4}
\le \varepsilon.
\]
Thus
\[
\cE_\mu(f)(\varepsilon)
\le M+1
=
1+\left\lceil\frac{2\ell}{\varepsilon}\right\rceil .
\]
\end{proof}

\begin{proposition}[Orientation-reversing interval homeomorphisms]
\label{prop:interval-reversing-homeo}
Let \(f:I\to I\) be an orientation-reversing homeomorphism. Then for every
\(x\in I\) there exists \(p_x\in I\) satisfying \(f^2(p_x)=p_x\) such that
\[
\emp_n(x)\wto \mu_x,
\]
where
\[
\mu_x=
\begin{cases}
\delta_{p_x}, & \text{if }f(p_x)=p_x,\\[0.3em]
\dfrac12(\delta_{p_x}+\delta_{f(p_x)}),
& \text{if }f(p_x)\neq p_x.
\end{cases}
\]
Moreover, for every Borel probability \(\mu\) on \(I\), and for every
\(\varepsilon>0\),
\[
\cE_\mu(f)(\varepsilon)
\le
\left(1+\left\lceil\frac{2\ell}{\varepsilon}\right\rceil\right)^2.
\]
In particular, \(f\) has low emergence.
\end{proposition}

\begin{proof}
Set \(g:=f^2\). Then \(g:I\to I\) is an orientation-preserving
homeomorphism. By Proposition~\ref{prop:interval-preserving-homeo}, for
every \(x\in I\) there exists \(p_x\in\operatorname{Fix}(g)\) such that
\[
g^n(x)=f^{2n}(x)\longrightarrow p_x.
\]
By continuity,
\[
f^{2n+1}(x)=f(f^{2n}(x))\longrightarrow f(p_x).
\]
Since \(p_x\in\operatorname{Fix}(g)\), we have \(f^2(p_x)=p_x\).

If \(f(p_x)=p_x\), then the whole orbit converges to \(p_x\), and hence
\(\emp_n(x)\wto\delta_{p_x}\). If \(f(p_x)\neq p_x\), then the even iterates
converge to \(p_x\) and the odd iterates converge to \(f(p_x)\). Therefore,
for every \(\varphi\in C^0(I)\), Lemma~\ref{lem:cesaro-homeo} gives
\[
\frac1n\sum_{k=0}^{n-1}\varphi(f^k(x))
\longrightarrow
\frac12\bigl(\varphi(p_x)+\varphi(f(p_x))\bigr).
\]
Thus
\[
\emp_n(x)\wto
\frac12(\delta_{p_x}+\delta_{f(p_x)}).
\]

Now let
\[
M:=\left\lceil\frac{2\ell}{\varepsilon}\right\rceil
\]
and take the grid \(y_j=j\ell/M\), \(j=0,\dots,M\). Consider the family
\[
\mathcal N
:=
\left\{
\frac12(\delta_{y_i}+\delta_{y_j}):0\le i,j\le M
\right\}.
\]
It has cardinality at most \((M+1)^2\). For any pair \(a,b\in I\), choose
grid points \(y_i,y_j\) with
\[
|a-y_i|\le \frac{\ell}{2M},
\qquad
|b-y_j|\le \frac{\ell}{2M}.
\]
Using the coupling matching \(a\) to \(y_i\) and \(b\) to \(y_j\), we get
\[
W_1\left(
\frac12(\delta_a+\delta_b),
\frac12(\delta_{y_i}+\delta_{y_j})
\right)
\le
\frac12|a-y_i|+\frac12|b-y_j|
\le
\frac{\ell}{2M}
\le
\frac{\varepsilon}{4}.
\]
The rest of the argument is identical to the end of
Proposition~\ref{prop:interval-preserving-homeo}: since every empirical
measure converges to a measure of the form above, the family \(\mathcal N\)
is enough in the definition of emergence, and
\[
\cE_\mu(f)(\varepsilon)
\le
|\mathcal N|
\le
(M+1)^2.
\]
\end{proof}

\subsection{Orientation-preserving circle homeomorphisms: irrational rotation number}

We recall the following classical fact.

\begin{theorem}[Unique ergodicity for irrational circle homeomorphisms]
\label{thm:irrational-circle-unique-ergodic}
Let \(f:\mathbb S^1\to\mathbb S^1\) be an orientation-preserving
homeomorphism with irrational rotation number. Then \(f\) is uniquely
ergodic.
\end{theorem}

\begin{proposition}[Low emergence in the irrational case]
\label{prop:circle-irrational-low}
Let \(f:\mathbb S^1\to\mathbb S^1\) be an orientation-preserving
homeomorphism with irrational rotation number. Then for every Borel
probability \(\mu\) on \(\mathbb S^1\),
\[
\cE_\mu(f)(\varepsilon)\le 1
\qquad\forall \varepsilon>0.
\]
In particular, \(f\) has low emergence.
\end{proposition}

\begin{proof}
By Theorem~\ref{thm:irrational-circle-unique-ergodic}, \(f\) has a unique
invariant probability measure, say \(\mu_f\). Unique ergodicity implies that
for every \(\varphi\in C^0(\mathbb S^1)\),
\[
\frac1n\sum_{k=0}^{n-1}\varphi(f^k(x))
\longrightarrow
\int\varphi\,d\mu_f
\]
uniformly in \(x\). Equivalently,
\[
\emp_n(x)\wto \mu_f
\qquad\text{for every }x\in\mathbb S^1.
\]
Since \(W_1\) induces the weak-\(\ast\) topology on \(\cP(\mathbb S^1)\), we
have
\[
W_1(\emp_n(x),\mu_f)\to0
\qquad\text{for every }x.
\]
Taking the single measure \(\nu_1:=\mu_f\) in the definition of emergence and
using dominated convergence gives
\[
\lim_{n\to\infty}
\int_{\mathbb S^1}W_1(\emp_n(x),\mu_f)\,d\mu(x)=0.
\]
Therefore \(\cE_\mu(f)(\varepsilon)\le1\) for every \(\varepsilon>0\).
\end{proof}

\subsection{Orientation-preserving circle homeomorphisms: rational rotation number}

Let \(f:\mathbb S^1\to\mathbb S^1\) be an orientation-preserving
homeomorphism with rational rotation number
\[
\rho(f)=\frac pq\in\mathbb Q
\]
in lowest terms. We fix a lift \(F:\mathbb R\to\mathbb R\) of \(f\), and set $H:=F^q-p$. Then \(H\) is a lift of $g:=f^q$ and has rotation number \(0\). In particular, \(H\) has fixed points.

\begin{lemma}[Lifted endpoint map]
\label{lem:lifted-endpoint-map}
For every \(t\in\mathbb R\), the limit
\[
\widetilde z(t):=\lim_{n\to\infty}H^n(t)
\]
exists. Moreover:
\begin{enumerate}[label=(\roman*)]
\item \(\widetilde z:\mathbb R\to\mathbb R\) is nondecreasing;
\item \(\widetilde z(t+1)=\widetilde z(t)+1\) for every \(t\in\mathbb R\);
\item \(\widetilde z\circ F=F\circ \widetilde z\).
\end{enumerate}
\end{lemma}

\begin{proof}
Since \(H\) is increasing, \(H(t+1)=H(t)+1\), and \(H\) has fixed points,
the set \(\operatorname{Fix}(H)\) is nonempty, closed, and invariant by
integer translations.

On each connected component \((a,b)\) of
\(\mathbb R\setminus\operatorname{Fix}(H)\), the continuous function
\(H(t)-t\) has no zero and hence has constant sign. If \(H(t)>t\) on
\((a,b)\), then \(H^n(t)\) is increasing and bounded above by \(b\), hence
converges to \(b\). If \(H(t)<t\), then \(H^n(t)\) decreases and converges to
\(a\). If \(t\in\operatorname{Fix}(H)\), then \(H^n(t)=t\). This proves the
existence of \(\widetilde z(t)\).

Since each \(H^n\) is increasing, the pointwise limit \(\widetilde z\) is
nondecreasing. Since \(H(t+1)=H(t)+1\), we have
\[
H^n(t+1)=H^n(t)+1
\]
for every \(n\), and therefore
\[
\widetilde z(t+1)=\widetilde z(t)+1.
\]

Finally, \(H\) commutes with \(F\). Indeed,
\[
H\circ F=F^{q+1}-p
\]
and
\[
F\circ H(t)=F(F^q(t)-p)=F^{q+1}(t)-p,
\]
because \(p\in\mathbb Z\) and \(F(s-p)=F(s)-p\). Hence
\[
H^n(F(t))=F(H^n(t))
\]
for every \(n\). Taking limits and using continuity of \(F\), we get $\widetilde z(F(t))=F(\widetilde z(t))$. \end{proof}

The map \(\widetilde z\) descends to a Borel map $z:\mathbb S^1\to\mathbb S^1$ defined by
\[
z(\pi(t)):=\pi(\widetilde z(t)).
\]
This is well-defined by item (ii) of Lemma~\ref{lem:lifted-endpoint-map}.
Moreover,
\[
z\circ f=f\circ z.
\]

Define $\Psi:\mathbb S^1\to\cP(\mathbb S^1)$ by
\[
\Psi(x):=\frac1q\sum_{j=0}^{q-1}\delta_{f^j(z(x))}.
\]

\begin{lemma}[Empirical convergence in the rational case]
\label{lem:rational-empirical-convergence}
For every \(x\in\mathbb S^1\),
\[
\emp_n(x)\wto \Psi(x).
\]
\end{lemma}

\begin{proof}
Fix \(x=\pi(t)\). Since
\[
\pi(H^n(t))=g^n(x)=f^{qn}(x),
\]
we have
\[
f^{qn}(x)\longrightarrow z(x).
\]
Therefore, for every \(r\in\{0,\dots,q-1\}\),
\[
f^{qn+r}(x)
=
f^r(f^{qn}(x))
\longrightarrow
f^r(z(x)).
\]
Let \(\varphi\in C^0(\mathbb S^1)\). Decompose the Birkhoff sum into residue
classes modulo \(q\):
\[
\frac1n\sum_{k=0}^{n-1}\varphi(f^k(x))
=
\sum_{r=0}^{q-1}
\frac{N_r(n)}{n}
\left(
\frac1{N_r(n)}
\sum_{j=0}^{N_r(n)-1}
\varphi(f^{qj+r}(x))
\right),
\]
where \(N_r(n)\) is the number of integers \(0\le k<n\) with
\(k\equiv r\pmod q\). Then
\[
\frac{N_r(n)}{n}\longrightarrow \frac1q.
\]
Moreover, by Lemma~\ref{lem:cesaro-homeo},
\[
\frac1{N_r(n)}
\sum_{j=0}^{N_r(n)-1}
\varphi(f^{qj+r}(x))
\longrightarrow
\varphi(f^r(z(x))).
\]
Thus
\[
\frac1n\sum_{k=0}^{n-1}\varphi(f^k(x))
\longrightarrow
\frac1q\sum_{r=0}^{q-1}\varphi(f^r(z(x))).
\]
This is precisely
\[
\emp_n(x)\wto \frac1q\sum_{r=0}^{q-1}\delta_{f^r(z(x))}
=\Psi(x).
\]
\end{proof}

We now prove the finite-length estimate for the image of \(\Psi\).

\begin{lemma}[A coupling estimate]
\label{lem:rational-coupling}
For every \(x,y\in\mathbb S^1\),
\[
W_1(\Psi(x),\Psi(y))
\le
\frac1q\sum_{j=0}^{q-1}
d_{\mathbb S^1}\bigl(f^j(z(x)),f^j(z(y))\bigr).
\]
\end{lemma}

\begin{proof}
Use the coupling that pairs the \(j\)-th atom of \(\Psi(x)\) with the
\(j\)-th atom of \(\Psi(y)\):
\[
\varpi
:=
\frac1q\sum_{j=0}^{q-1}
\delta_{(f^j(z(x)),\,f^j(z(y)))}.
\]
Then \(\varpi\) is a coupling of \(\Psi(x)\) and \(\Psi(y)\), so
\[
W_1(\Psi(x),\Psi(y))
\le
\int d_{\mathbb S^1}\,d\varpi
=
\frac1q\sum_{j=0}^{q-1}
d_{\mathbb S^1}\bigl(f^j(z(x)),f^j(z(y))\bigr).
\]
\end{proof}

\begin{lemma}[Length of monotone degree-one curves]
\label{lem:monotone-degree-one-length}
Let \(h:\mathbb R\to\mathbb R\) be nondecreasing and satisfy
\[
h(t+1)=h(t)+1.
\]
Define
\[
\eta:[0,1]\to\mathbb S^1,
\qquad
\eta(t):=\pi(h(t)).
\]
Then
\[
\operatorname{length}(\eta)\le 4D.
\]
\end{lemma}

\begin{proof}
Fix a partition
\[
0=t_0<t_1<\cdots<t_m=1.
\]
Set
\[
\Delta_j:=h(t_{j+1})-h(t_j).
\]
Since \(h\) is nondecreasing, \(\Delta_j\ge0\). Also,
\[
\sum_{j=0}^{m-1}\Delta_j=h(1)-h(0)=1.
\]
Thus each \(\Delta_j\in[0,1]\). Since the circle has diameter \(D\),
\[
d_{\mathbb S^1}(\eta(t_j),\eta(t_{j+1}))
\le
2D\min\{\Delta_j,1-\Delta_j\}.
\]
Therefore
\[
\sum_{j=0}^{m-1}
d_{\mathbb S^1}(\eta(t_j),\eta(t_{j+1}))
\le
2D
\sum_{j=0}^{m-1}
\min\{\Delta_j,1-\Delta_j\}.
\]
Let
\[
A:=\{j:\Delta_j\le1/2\},
\qquad
B:=\{j:\Delta_j>1/2\}.
\]
Then \(B\) has at most one element, since \(\sum_j\Delta_j=1\). Hence
\[
\sum_j\min\{\Delta_j,1-\Delta_j\}
=
\sum_{j\in A}\Delta_j+\sum_{j\in B}(1-\Delta_j)
\le
1+|B|
\le2.
\]
Thus every partition has length at most \(4D\), and taking the supremum over
partitions gives $\operatorname{length}(\eta)\le4D$.
\end{proof}

\begin{proposition}[Finite length of the rational limit curve]
\label{prop:rational-limit-curve-length}
Let $\Gamma:[0,1]\to\cP(\mathbb S^1)$ be the curve given by
$$
\Gamma(t):=\Psi(\pi(t)).
$$
Then, $\operatorname{length}(\Gamma)\le4D$.
\end{proposition}

\begin{proof}
Fix a partition $0=t_0<t_1<\cdots<t_m=1$. By Lemma~\ref{lem:rational-coupling},
\[
\sum_{j=0}^{m-1}W_1(\Gamma(t_j),\Gamma(t_{j+1}))
\le
\frac1q\sum_{i=0}^{q-1}
\sum_{j=0}^{m-1}
d_{\mathbb S^1}
\bigl(
f^i(z(\pi(t_j))),
f^i(z(\pi(t_{j+1})))
\bigr).
\]
Fix \(i\in\{0,\dots,q-1\}\). Define $h_i(t):=F^i(\widetilde z(t))$. Since both \(F^i\) and \(\widetilde z\) are nondecreasing, \(h_i\) is
nondecreasing. Moreover,
\[
h_i(t+1)=h_i(t)+1.
\]
Indeed, \(\widetilde z(t+1)=\widetilde z(t)+1\), and \(F^i(s+1)=F^i(s)+1\).
Also,
\[
\pi(h_i(t))
=
f^i(z(\pi(t))).
\]
Therefore, by Lemma~\ref{lem:monotone-degree-one-length},
\[
\sum_{j=0}^{m-1}
d_{\mathbb S^1}
\bigl(
f^i(z(\pi(t_j))),
f^i(z(\pi(t_{j+1})))
\bigr)
\le
4D.
\]
Averaging over \(i=0,\dots,q-1\), we obtain
\[
\sum_{j=0}^{m-1}W_1(\Gamma(t_j),\Gamma(t_{j+1}))
\le
4D.
\]
Taking the supremum over all partitions proves the proposition.
\end{proof}

\begin{lemma}[Covering a finite-length image]
\label{lem:finite-length-cover}
Let \((Y,d)\) be a metric space and let \(\gamma:[0,1]\to Y\) be a map with
finite length \(L\). Then, for every \(\varepsilon>0\), the image
\(\gamma([0,1])\) admits an \(\varepsilon\)-net with at most
\[
1+\left\lceil\frac{L}{\varepsilon}\right\rceil
\]
points.
\end{lemma}

\begin{proof}
It is enough to prove that every \(\varepsilon\)-separated subset of
\(\gamma([0,1])\) has cardinality at most \(1+L/\varepsilon\).

Let $\{\gamma(s_0),\dots,\gamma(s_m)\}$ be an \(\varepsilon\)-separated subset of \(\gamma([0,1])\). Reordering the parameters, we may assume
\[
0\le s_0<s_1<\cdots<s_m\le1.
\]
Then
\[
L
\ge
\sum_{j=0}^{m-1}d(\gamma(s_j),\gamma(s_{j+1}))
\ge
m\varepsilon,
\]
and thus $m+1\le 1+\frac{L}{\varepsilon}$.
If no \(\varepsilon\)-net with at most
\(1+\lceil L/\varepsilon\rceil\) points existed, one could construct an
\(\varepsilon\)-separated subset with larger cardinality, contradicting the
estimate above.
\end{proof}

\begin{proposition}[Low emergence in the rational case]
\label{prop:circle-rational-low}
Let \(f:\mathbb S^1\to\mathbb S^1\) be an orientation-preserving
homeomorphism with rational rotation number. Then for every Borel probability
\(\mu\) on \(\mathbb S^1\), and for every \(\varepsilon>0\),
\[
\cE_\mu(f)(\varepsilon)
\le
1+\left\lceil\frac{4D}{\varepsilon}\right\rceil.
\]
In particular, \(f\) has low emergence.
\end{proposition}

\begin{proof}
By Proposition~\ref{prop:rational-limit-curve-length} and
Lemma~\ref{lem:finite-length-cover}, there exist measures $\nu_1,\dots,\nu_N\in\Gamma([0,1])$ such that
\[
N\le 1+\left\lceil\frac{4D}{\varepsilon}\right\rceil
\]
and
\[
\Gamma([0,1])
\subset
\bigcup_{i=1}^N B_{W_1}(\nu_i,\varepsilon).
\]
Equivalently, for every \(x\in\mathbb S^1\), $\min_{1\le i\le N}W_1(\Psi(x),\nu_i)\le\varepsilon$.

Set $g_n(x):=
\min_{1\le i\le N}W_1(\emp_n(x),\nu_i)$. By Lemma~\ref{lem:rational-empirical-convergence}, \(\emp_n(x)\wto\Psi(x)\)
for every \(x\), hence also \(W_1(\emp_n(x),\Psi(x))\to0\). Therefore, $\limsup_{n\to\infty}g_n(x)
\le
\varepsilon$, for every $x$. Since \(0\le g_n\le D\), applying Fatou to \(D-g_n\) gives
\[
\limsup_{n\to\infty}
\int_{\mathbb S^1}g_n(x)\,d\mu(x)
\le
\int_{\mathbb S^1}\limsup_{n\to\infty}g_n(x)\,d\mu(x)
\le
\varepsilon.
\]
Thus, the measures \(\nu_1,\dots,\nu_N\) are admissible in the definition of
emergence, and
\[
\cE_\mu(f)(\varepsilon)
\le
N
\le
1+\left\lceil\frac{4D}{\varepsilon}\right\rceil.
\]\end{proof}

\subsection{Orientation-reversing circle homeomorphisms}

\begin{proposition}[Empirical convergence for orientation-reversing circle homeomorphisms]
\label{prop:circle-reversing-empirical}
Let \(f:\mathbb S^1\to\mathbb S^1\) be an orientation-reversing
homeomorphism. Then for every \(x\in\mathbb S^1\) there exists
\(p_x\in\mathbb S^1\) satisfying
\[
f^2(p_x)=p_x
\]
such that
\[
\emp_n(x)\wto \mu_x,
\]
where
\[
\mu_x=
\begin{cases}
\delta_{p_x}, & \text{if }f(p_x)=p_x,\\[0.3em]
\dfrac12(\delta_{p_x}+\delta_{f(p_x)}),
& \text{if }f(p_x)\neq p_x.
\end{cases}
\]
\end{proposition}

\begin{proof}
Set $g:=f^2$. Then \(g\) is an orientation-preserving circle homeomorphism. We first note that \(f\) has a fixed point. Indeed, let
\(F:\mathbb R\to\mathbb R\) be a lift of \(f\). Since \(f\) is
orientation-reversing, we may choose \(F\) so that
\[
F(t+1)=F(t)-1.
\]
The continuous function
\[
A(t):=F(t)-t
\]
satisfies
\[
A(t+1)=A(t)-2.
\]
Hence \(A([0,1])\) contains an integer \(m\). If \(A(t_0)=m\), then
\[
F(t_0)=t_0+m,
\]
so \(\pi(t_0)\) is a fixed point of \(f\). In particular, \(g=f^2\) has a
fixed point.

Thus \(g\) is orientation-preserving with rotation number \(0\). Applying the
rational case to \(g\), with denominator \(q=1\), we get that for every
\(x\in\mathbb S^1\) there exists \(p_x\in\operatorname{Fix}(g)\) such that
\[
g^n(x)=f^{2n}(x)\longrightarrow p_x.
\]
By continuity,
\[
f^{2n+1}(x)=f(f^{2n}(x))\longrightarrow f(p_x).
\]
Since \(p_x\in\operatorname{Fix}(g)\), we have \(f^2(p_x)=p_x\).

If \(f(p_x)=p_x\), then both even and odd iterates converge to \(p_x\), and
therefore \(\emp_n(x)\wto\delta_{p_x}\). If \(f(p_x)\neq p_x\), then the
even iterates converge to \(p_x\) and the odd iterates converge to \(f(p_x)\).
For every \(\varphi\in C^0(\mathbb S^1)\), Lemma~\ref{lem:cesaro-homeo}
gives
\[
\frac1n\sum_{k=0}^{n-1}\varphi(f^k(x))
\longrightarrow
\frac12\bigl(\varphi(p_x)+\varphi(f(p_x))\bigr).
\]
Thus, $\emp_n(x)\wto
\frac12(\delta_{p_x}+\delta_{f(p_x)})$.
\end{proof}

\begin{proposition}[Low emergence for orientation-reversing circle homeomorphisms]\label{prop:circle-reversing-low}
Let \(f:\mathbb S^1\to\mathbb S^1\) be an orientation-reversing
homeomorphism. Then for every Borel probability \(\mu\) on \(\mathbb S^1\),
and for every \(\varepsilon>0\),
\[
\cE_\mu(f)(\varepsilon)
\le
\left(1+\left\lceil\frac{2D}{\varepsilon}\right\rceil\right)^2.
\]
In particular, \(f\) has low emergence.
\end{proposition}

\begin{proof}
Let $M:=\left\lceil\frac{2D}{\varepsilon}\right\rceil$
and choose \(M+1\) equally spaced points $y_0,\dots,y_M\in\mathbb S^1$. Since \(D=1/2\), every point of \(\mathbb S^1\) is within distance
\[
\frac{1}{2(M+1)}
=
\frac{D}{M+1}
\le
\frac{\varepsilon}{2}
\]
of some \(y_j\).

Consider
\[
\mathcal N
:=
\left\{
\frac12(\delta_{y_i}+\delta_{y_j}):
0\le i,j\le M
\right\}.
\]
Then, $|\mathcal N|\le(M+1)^2$.

By Proposition~\ref{prop:circle-reversing-empirical}, for every
\(x\in\mathbb S^1\), the empirical measures converge to a probability
measure \(\mu_x\) which is either a Dirac mass or the average of a
\(2\)-cycle. If $\mu_x=\frac12(\delta_a+\delta_b)$, where we allow \(a=b\) in the Dirac case, choose \(y_i,y_j\) such that
\[
d(a,y_i)\le\frac{\varepsilon}{2},
\qquad
d(b,y_j)\le\frac{\varepsilon}{2}.
\]
Then $\nu:=\frac12(\delta_{y_i}+\delta_{y_j})
\in\mathcal N$, and, using the coupling which sends \(a\) to \(y_i\) and \(b\) to
\(y_j\), we get
\[
W_1(\mu_x,\nu)
\le
\frac12 d(a,y_i)+\frac12 d(b,y_j)
\leq
\frac{\varepsilon}{2}.
\]

Therefore, $\min_{\nu\in\mathcal N}W_1(\mu_x,\nu)
\le
\frac{\varepsilon}{2}$ for every $x\in\mathbb S^1$.

Now, arguing exactly as in the proof of
Proposition~\ref{prop:circle-rational-low}, the family \(\mathcal N\) is
admissible in the definition of emergence, and
\[
\cE_\mu(f)(\varepsilon)
\le
|\mathcal N|
\le
(M+1)^2
=
\left(
1+\left\lceil\frac{2D}{\varepsilon}\right\rceil
\right)^2.
\]
\end{proof}

\subsection{Conclusion for circle homeomorphisms}

\begin{theorem}[Low emergence for circle homeomorphisms]
\label{thm:circle-homeomorphisms-low}
Let \(f:\mathbb S^1\to\mathbb S^1\) be a homeomorphism. Then, for every
Borel probability \(\mu\) on \(\mathbb S^1\), \(f\) has low emergence with respect to \(\mu\).
\end{theorem}

\begin{proof}
If \(f\) is orientation-preserving, then either its rotation number is
irrational or rational. In the irrational case, the result follows from
Proposition~\ref{prop:circle-irrational-low}. In the rational case, it
follows from Proposition~\ref{prop:circle-rational-low}.

If \(f\) is orientation-reversing, the result follows from
Proposition~\ref{prop:circle-reversing-low}.
\end{proof}

From now on, in order to prove Theorem~\ref{main}, we may assume
\[
|\deg f|\ge2.
\]
Indeed, if \(|\deg f|=1\), then a \(C^2\) immersion
\(f:\mathbb S^1\to\mathbb S^1\) is a \(C^2\) circle diffeomorphism, hence a
circle homeomorphism, and the conclusion follows from
Theorem~\ref{thm:circle-homeomorphisms-low}.

\section{Entropy formula and physical measures}

The following is an adaptation of one of the main theorems in \cite{CYZ} to our context. 

\begin{theorem}[Entropy formula for asymptotic measures]
\label{thm:entropyfor}
Let $f:S^1\to S^1$ be a $C^2$ immersion topologically conjugate to $x\mapsto dx$. Then, there exists a full Lebesgue measure subset
\[
\Gamma\subset S^1
\]
such that every asymptotic measure of every point $x\in\Gamma$ satisfies the entropy formula
\[
h_\mu(f)=\int \log |Df|\,d\mu.
\]
\end{theorem}

The goal of this section is to show how to deduce this result from \cite{CYZ} and explore some consequences about physicality of invariant measures. 

\subsection{Tools for the proof}

\subsubsection{A smooth realization of the natural extension}

Let us recall some well-known facts about natural extensions of non-invertible maps. Consider $\mathbb{S}^1_f\eqdef\{(x_j)_{j\in\Z}\subset\mathbb{S}^1;f(x_j)=x_{j+1}\}$ the natural extension of the endomorphism $f:\mathbb{S}^1\to\mathbb{S}^1$ and consider $\hat{f}:\mathbb{S}_f^1\to\mathbb{S}_f^1$ the left shift. Consider $\pi:\mathbb{S}_f^1\to\mathbb{S}^1$ the natural projection onto the $0$-th coordinate. Then, $\pi$ semiconjugates $\hat{f}$ and $f$. Moreover, $\pi_{\star}:\cP_{\hat{f}}(\mathbb{S}^1_f)\to\cP_f(\mathbb{S}^1)$ is a continuous bijection.  

We now recall a result proven in the appendix of \cite{VianaYang}. Let $U$ be an open neighborhood of $i(\mathbb{S}^1)$ in $\R^2$, where $i:\mathbb{S}^1\to\R^2$ is the natural immersion $i(x)=(\cos x,\sin x)$. For simplicity, we identify $U=\mathbb{S}^1\times(-1,1)$. Consider $g:U\to U$ defined by
\[
g(x,v)=(f(x),\lambda v+i(x)),
\]
for some $0<\lambda<1$. 

\begin{lemma}
    \label{lem:vianayang}
Let $\Lambda\eqdef\cap_{n\geq 0}g^n(U)$. Then, for $\lambda$ sufficiently small $\Lambda$ is an attractor with a dominated splitting $T_{\Lambda}U=E^s\oplus F$ with $\|Dg|_{E^s}\|<1$. Moreover, $g|_{\Lambda}$ is topologically conjugate to $\hat{f}$.
\end{lemma}

We shall denote by $h:\mathbb{S}^1_f\to\Lambda$ the conjugacy given by the above lemma. A proof can be found in the appendix of \cite{VianaYang}.

\subsubsection{Entropy formula for asymptotic measures}

Applying Theorem F from \cite{CYZ} to the map $g$ above, we obtain

\begin{lemma}
    \label{lem:cyz}
There exists a full Lebesgue measure $\tilde{\Gamma}\subset U$ such that for every $y\in\tilde{\Gamma}$ and every asymptotic measure $\nu\in\cA_g(y)$ it holds
\[
h_{\nu}(g)\geq\int\log\|Dg|_{F}\|d\nu
\]
\end{lemma}

\subsection{Proof of the entropy formula}
Given $x\in\T$ consider $\tg_x\eqdef\{v\in(-1,1);(x,v)\in\tg\}$, where $\tg$ is the full measure set given by Lemma~\ref{lem:cyz}. By Fubini, the set 
\[
\Gamma\eqdef\{x\in\T;\Leb_{(-1,1)}(\tg_x)=1\}
\]
has full measure in $\T$. 

Take $x\in\Gamma$ and $\mu\in\cA_f(x)$. Let $n_j\to\infty$ be such that $\emp_{n_j}^f(x)\to\mu$. Take $v\in\tg_x$ such that $y=(x,v)\in\tg$ and consider the sequence of empirical measures $e^g_{n_j}(y)$. Up to further extraction of a subsequence, we may assume that 
\[
\emp^g_{n_j}(y)\to\nu\in\cA_g(y).
\]
Observe that, since $\Lambda$ is an attractor, we must have $\nu$ supported in $\Lambda$. Thus, by Lemma~\ref{lem:cyz} the entropy formula holds for $\nu$. Let $p:U\to\T$ be the projection on the first coordinate. Notice, by definition, that $p\circ g=f\circ p$. Thus, for every $n>0$ it holds
\[
p_{\star}e^g_n(y)=e^f_n(x).
\]
Therefore, $p_{\star}\nu=\mu$. Observe that the diagram below is commutative.

\[
\begin{tikzcd}[column sep=4em, row sep=4em]
\Tf
  \arrow[r, "h"]
  \arrow[d, "\pi"']
&
\Lambda
  \arrow[d, "p"]
\\
\T
  \arrow[r, "I_d"']
&
\T
\end{tikzcd}
\]
This proves that $h_{\star}\hat{\mu}=\nu$, where $\hat{\mu}$ is the unique element in $\cP_{\hf}(\Tf)$ satisfying $\pi_{\star}\hat{\mu}=\mu$. 

Since, for $q=(z,v)\in U$ $Dp(g(q))Dg(q)=Df(z)Dp(q)$. Restricting both sides to the one-dimensional subspace $F(q)$ and taking determinants on both sides we get
\[
a(g(q))\|Dg(q)|_F\|=|Df(z)|a(q),
\]
where $a(q)=\det Dp(q)|_{F}$. Since $Dp(q): F(q)\to T_z\T$ is an isomorphism for every $q$, $a(q)\neq 0$ and thus
\[
\|Dg(q)|_F\|=|Df(z)|\frac{a(q)}{a(g(q))}
\]
Taking logarithms on both sides, integrating against $\nu$, using its $g$-invariance and performing the change of variables $z=p(q)$ we deduce immediately 
\[
\int\log\|Dg(q)|_F\|d\nu(q)=\int\log|Df(p(q))|d\nu(q)=\int\log|Df(z)|d\mu(z)
\]
Therefore, $h_{\mu}(f)=h_{\hat{\mu}}(\hf)=h_\nu(g)\geq\int\log|Df|d\mu$, as announced. Ruelle inequality gives the other direction. \qed

\subsection{Consequences on physicality}

\begin{lemma}[Ergodic decomposition]
\label{lem:decomposicao}
Let $\mu$ be an invariant probability measure satisfying
\[
h_\mu(f)=\int\log|Df|\,d\mu.
\]
Then, for $\hat\mu$-almost every ergodic component $\mu_\omega$ of the ergodic decomposition of $\mu$,
\[
h_{\mu_\omega}(f)=\int\log|Df|\,d\mu_\omega.
\]  
\end{lemma}
\begin{proof}
Since both the metric entropy and the Lyapunov integral are affine with respect to the ergodic decomposition, we have
\[
h_\mu(f)
=
\int h_{\mu_\omega}(f)\,d\eta(\omega)
\]
and
\[
\int\log|Df|\,d\mu
=
\int
\left(
\int\log|Df|\,d\mu_\omega
\right)
d\eta(\omega).
\]
Moreover, Ruelle's inequality yields
\[
h_{\mu_\omega}(f)
\le
\int\log|Df|\,d\mu_\omega
\]
for every ergodic component. Since equality holds after integration, the inequality must be an equality for $\hat\mu$-almost every component.
\end{proof}

Measures satisfying the entropy formula are relevant to study statistical behavior of orbits due to their physicality, as shown by Theorem VII.1.1 of \cite{quianetc}.

\begin{lemma}
    \label{lem:quian}
Let $\mu\in\cP^{\operatorname{erg}}_f(\T)$ be an ergodic invariant measure with 
\[
h_\mu(f)=\int\log|Df|d\mu>0.
\]
Then, $\mu$ is an absolutely continuous measure. In particular, $\mu$ is a physical measure. 
\end{lemma}

\section{Ma\~n\'e's consequences}	
Let $f\colon \Sone\to \Sone$ be a $C^2$ endomorphism without critical points. Equivalently, $f$ is a $C^2$ immersion of the circle. The following comes from Ma\~n\'e's paper \cite{mane}.

\subsection{Julia set, plateaux, and Ma\~n\'e's structural consequences}

We first record the bounded-period consequence of Ma\~n\'e's Theorem~C (Corollary~I just after Theorem~C).

\begin{corollary}[Mañé, Corollary~I after Theorem~C]
\label{cor:mane-bounded-periods}
Let $f$ be a $C^2$ endomorphism of the circle or interval, and let
$\Lambda$ be a compact set containing no critical points.
Then the periods of the sinks or non-hyperbolic periodic points whose
orbits are contained in $\Lambda$ are bounded.
\end{corollary}

Now assume that $f\colon \Sone\to \Sone$ is a circle immersion of degree $|d|\neq 1$.
Ma\~n\'e recalls that there exists a continuous map
\[
h\colon \Sone\to \Sone
\]
of degree $1$ such that
\[
h\circ f = f_d\circ h,
\qquad
f_d(z)=z^d.
\]
If $f$ is an immersion, then $h$ is monotone: for every $z\in \Sone$, the fiber
\[
h^{-1}(\{z\})
\]
is either a single point or a closed interval $[a,b]$ with $a\neq b$.
In the latter case, the open interval $(a,b)$ is called a \emph{plateau} of $f$.
Following Ma\~n\'e, we define the Julia set by
\[
J(f):=\Sone\setminus \bigcup\{\text{plateaux of }f\}.
\]

Using the semiconjugacy $h$, one obtains the following structural properties:
\begin{enumerate}[label=(\roman*), leftmargin=2.3em]
	\item two plateaux are either disjoint or coincide;
	\item $f$ maps plateaux diffeomorphically onto plateaux;
	\item every plateau is periodic, eventually periodic, or wandering;
	\item if $f$ is $C^2$, every plateau is periodic or eventually periodic.
\end{enumerate}
Moreover, for a $C^2$ immersion, the set of periodic plateaux is finite (this is Corollary~III of Theorem~C).

Now, there is the following measure-theoretic dichotomy is given by Ma\~n\'e \cite[Corollary~I of Theorem~D]{mane}.

\begin{corollary}\label{cor:mane-D}
	For a $C^2$ immersion of the circle,
	\[
	J(f)=\Sone
	\qquad\text{or}\qquad
	\Leb(J(f))=0.
	\]
\end{corollary}

And in the case $J(f)=\Sone$ one has the $0$--$1$ law \cite[Theorem~E]{mane}.

\begin{theorem}\label{thm:mane-E}
	If $f$ is a $C^2$ immersion of the circle and $J(f)=\Sone$, then every invariant Borel set has
	Lebesgue measure $0$ or $1$.
\end{theorem}
	
\begin{remark}\label{rem:mane-forward-invariant}
	We recall that, following Mañé, ``invariant'' is understood in the forward sense:
	\[
	f(A)\subset A.
	\]
	Thus, Corollary~\ref{cor:mane-bounded-periods}, Theorem~\ref{thm:mane-E}, as well as the other consequences we use
from \cite{mane}, apply to forward invariant sets. In the concrete applications below, the sets we consider are often fully invariant, but this stronger property will not be needed.
\end{remark}

\subsection{Passing to an iterate}

Before treating the two cases arising from Mañé's dichotomy, we record
some simple facts relating iterates of the dynamics and emergence.
They will be used both in the plateaux case and later in the
$J(f)=\Sone$ case.

Let \(N\ge1\), set \(F:=f^N\), and define
\[
T_N:\cP(\Sone)\to\cP(\Sone),
\qquad
T_N(\eta):=
\frac1N\sum_{j=0}^{N-1}(f^j)_*\eta.
\]

\begin{proposition}\label{prop:iterate}
For every \(x\in\Sone\),
\[
\cA_f(x)=T_N(\cA_F(x)).
\]
\end{proposition}

\begin{proof}
For every \(q\ge1\),
\[
\emp_q^F(x)
=
\frac1q\sum_{k=0}^{q-1}\delta_{f^{kN}(x)}.
\]
Therefore,
\[
\begin{aligned}
T_N(\emp_q^F(x))
&=
\frac1N\sum_{j=0}^{N-1}
(f^j)_*
\left(
\frac1q\sum_{k=0}^{q-1}\delta_{f^{kN}(x)}
\right)
\\
&=
\frac1{qN}
\sum_{j=0}^{N-1}\sum_{k=0}^{q-1}
\delta_{f^{j+kN}(x)}
\\
&=
\frac1{qN}\sum_{m=0}^{qN-1}\delta_{f^m(x)}
=
\emp_{qN}^f(x).
\end{aligned}
\]
Thus,
\[
\emp_{qN}^f(x)=T_N(\emp_q^F(x)).
\]

Now write \(m=qN+r\), with \(0\le r<N\). Then
\[
\emp_{qN+r}^f(x)
=
\frac{qN}{qN+r}\,\emp_{qN}^f(x)
+
\frac1{qN+r}
\sum_{j=qN}^{qN+r-1}\delta_{f^j(x)},
\]
and hence
\[
\Wone(\emp_{qN+r}^f(x),\emp_{qN}^f(x))
\le
\frac{2r\,\diam(\Sone)}{qN+r}
\le
\frac{2N\,\diam(\Sone)}{qN+r}.
\]
In particular, this quantity tends to zero as \(q\to\infty\), uniformly
in \(x\) and \(0\le r<N\).

Let \(\eta\in\cA_F(x)\). There exists \(q_k\to\infty\) such that
\[
\emp_{q_k}^F(x)\wto\eta.
\]
Since \(T_N\) is continuous,
\[
\emp_{q_kN}^f(x)
=
T_N(\emp_{q_k}^F(x))
\wto T_N(\eta),
\]
and therefore
\[
T_N(\eta)\in\cA_f(x).
\]
Hence
\[
T_N(\cA_F(x))\subset\cA_f(x).
\]

Conversely, let \(\nu\in\cA_f(x)\). There exists \(m_k\to\infty\) such that
\[
\emp_{m_k}^f(x)\wto\nu.
\]
Write
\[
m_k=q_kN+r_k,
\qquad
0\le r_k<N.
\]
By the estimate above,
\[
\Wone(\emp_{m_k}^f(x),\emp_{q_kN}^f(x))\to0,
\]
and hence
\[
\emp_{q_kN}^f(x)\wto\nu.
\]
Since \(\cP(\Sone)\) is compact, after passing to a subsequence we may
assume that
\[
\emp_{q_k}^F(x)\wto\eta
\]
for some \(\eta\in\cA_F(x)\). Therefore,
\[
\nu
=
\lim_{k\to\infty}\emp_{q_kN}^f(x)
=
\lim_{k\to\infty}T_N(\emp_{q_k}^F(x))
=
T_N(\eta).
\]
Thus
\[
\cA_f(x)\subset T_N(\cA_F(x)),
\]
which proves the equality.
\end{proof}

\begin{corollary}\label{cor:iterate-convergence}
Let
\[
C_f:=\{x\in\Sone:\emp_n^f(x)\text{ converges}\},
\qquad
C_F:=\{x\in\Sone:\emp_n^F(x)\text{ converges}\}.
\]
Then
\[
C_F\subset C_f.
\]
\end{corollary}

\begin{proof}
If \(x\in C_F\), then \(\cA_F(x)\) is a singleton. By
Proposition~\ref{prop:iterate}, \(\cA_f(x)\) is also a singleton.
Since \(\cP(\Sone)\) is compact, it follows that
\((\emp_n^f(x))_{n\ge1}\) converges. Hence \(x\in C_f\).
\end{proof}

\begin{corollary}\label{cor:iterate-emergence}
Let \(\mu\) be a Borel probability measure on \(\Sone\), and set
\[
L_N:=
\frac1N\sum_{j=0}^{N-1}\Lip(f^j).
\]
Then, for every \(\varepsilon>0\),
\[
\cE_\mu(f)(\varepsilon)
\le
\cE_\mu\left(F,\frac{\varepsilon}{L_N}\right).
\]
In particular, if \(F\) has low emergence with respect to \(\mu\), then
so does \(f\).
\end{corollary}

\begin{proof}
For every \(\alpha,\beta\in\cP(\Sone)\),
\[
\begin{aligned}
\Wone(T_N(\alpha),T_N(\beta))
&\le
\frac1N\sum_{j=0}^{N-1}
\Wone((f^j)_*\alpha,(f^j)_*\beta)
\\
&\le
L_N\Wone(\alpha,\beta).
\end{aligned}
\]
Thus \(T_N\) is \(L_N\)-Lipschitz.

Fix \(\varepsilon>0\), and let
\[
M:=
\cE_\mu\left(F,\frac{\varepsilon}{L_N}\right).
\]
Choose measures \(\nu_1,\dots,\nu_M\in\cP(\Sone)\) admissible in the
definition of this emergence. Let \(m=qN+r\), with \(0\le r<N\). Using $\emp_{qN}^f(x)=T_N(\emp_q^F(x))$, we obtain
\[
\begin{aligned}
\min_{1\le i\le M}
\Wone(\emp_m^f(x),T_N(\nu_i))
&\le
\Wone(\emp_m^f(x),\emp_{qN}^f(x))
\\
&\quad+
\min_{1\le i\le M}
\Wone(T_N(\emp_q^F(x)),T_N(\nu_i))
\\
&\le
\frac{2N\,\diam(\Sone)}{m}
+
L_N
\min_{1\le i\le M}
\Wone(\emp_q^F(x),\nu_i).
\end{aligned}
\]
Integrating with respect to \(\mu\) and taking the upper limit as
\(m\to\infty\), we obtain
\[
\limsup_{m\to\infty}
\int_{\Sone}
\min_{1\le i\le M}
\Wone(\emp_m^f(x),T_N(\nu_i))
\,d\mu(x)
\le
\varepsilon.
\]
Hence, $\cE_\mu(f)(\varepsilon)
\le
\cE_\mu\left(F,\frac{\varepsilon}{L_N}\right)$. The last assertion follows directly.

\end{proof}

\begin{corollary}\label{cor:iterate-compact-family}
Let \(K\subset\cP(\Sone)\) be compact. If, for some \(x\in\Sone\), $\cA_F(x)\subset K$, then
\[
\cA_f(x)\subset T_N(K).
\]
Moreover, if \(K\) has polynomial \(\Wone\)-covering numbers, then so
does \(T_N(K)\).
\end{corollary}

\begin{proof}
The first assertion follows from Proposition~\ref{prop:iterate}.
The second follows from the fact that \(T_N\) is Lipschitz, as shown in
the proof of Corollary~\ref{cor:iterate-emergence}.
\end{proof}
	
\subsection{The plateaux case}
	In the task of proving Theorem \ref{main}, we first show that, in the case $f$ admits plateaux, the emergence is low.
	
	\begin{proposition}\label{prop:plateau}
		Let $f\colon \Sone\to \Sone$ be a $C^2$ endomorphism without critical points.
		Assume $\Leb(J(f))=0$.
		Then $f$ has low emergence with respect to the Lebesgue measure.
	\end{proposition}
	
	\begin{proof}
        Since $\Leb(J(f))=0$, the complement $\Sone\setminus J(f)$ has full Lebesgue measure. By definition of $J(f)$, every point of this complement belongs to a plateau. Moreover, by property~(XIV) after Theorem~C in \cite{mane}, every plateau of a $C^2$ immersion is periodic or eventually periodic. Finally, by Corollary~III after Theorem~C in \cite{mane}, the periodic plateaux are finite in number.
		
Hence, Lebesgue-a.e.\ point is eventually mapped into one of finitely many periodic plateau cycles
	\[
	I_0\to I_1\to \cdots \to I_{q-1}\to I_0.
	\]
By property~(XII) in \cite{mane}, \(f\) maps every plateau
diffeomorphically onto a plateau. Hence, along each periodic plateau cycle
\[
I_0\to I_1\to \cdots \to I_{q-1}\to I_0,
\]
the first return map
\[
g:=f^q|_{I_0}:I_0\to I_0
\]
is a \(C^2\)-diffeomorphism of the open interval \(I_0\).

Writing \(I_0=(a_0,b_0)\), this map extends to a homeomorphism
\[
g:=f^q|_{\overline{I_0}}:\overline{I_0}\to\overline{I_0}.
\]
Indeed, \(f^q(I_0)=I_0\), and by continuity
\[
f^q(\overline{I_0})\subset \overline{I_0}.
\]
Since \(f^q|_{I_0}\) is a diffeomorphism onto \(I_0\), the endpoints must
be mapped to endpoints, so
\[
f^q(\{a_0,b_0\})=\{a_0,b_0\}.
\]
Thus \(g\) is an interval homeomorphism of the compact interval
\(\overline{I_0}\).

If $g$ is orientation-preserving, Proposition~\ref{prop:interval-preserving-homeo} shows that every empirical measure
of $g$ converges to a Dirac mass. Hence all possible limits belong to
\[
K_g:=\{\delta_p:p\in\overline{I_0}\}.
\]
This is a compact subset of $\cP(\overline{I_0})$ with
$O(\varepsilon^{-1})$ $W_1$-covering numbers.

If $g$ is orientation-reversing, Proposition~\ref{prop:interval-reversing-homeo} shows that every empirical measure converges to a probability supported on a periodic orbit of cardinality at most two. Hence, all possible limits belong to
\[
K_g:=
\left\{
\frac12(\delta_a+\delta_b):
a,b\in\overline{I_0}
\right\},
\]
which is compact and has $O(\varepsilon^{-2})$ $W_1$-covering numbers.

For \(x\in I_0\), the \(f^q\)-orbit of \(x\) coincides with the
\(g\)-orbit of \(x\). Hence
\[
\cA_{f^q}(x)=\cA_g(x)\subset K_g.
\]
Applying Corollary~\ref{cor:iterate-compact-family} with \(N=q\), we obtain
\[
\cA_f(x)\subset T_q(K_g),
\]
and \(T_q(K_g)\) has polynomial \(W_1\)-covering numbers.

Since removing a finite initial segment of an orbit does not change its
asymptotic measures, the same conclusion holds for every plateau that is
eventually mapped into this periodic cycle.

As there are only finitely many periodic plateau cycles, the union of the corresponding families still has polynomial covering numbers.
Proposition~\ref{prop:compact-family-bound} finishes the proof.
	\end{proof}
	
	\begin{remark}
		So, from now on, the only genuinely difficult branch is the $J(f)=\Sone$ setting.
	\end{remark}
\section{The \texorpdfstring{$J(f)=\Sone$}{J(f)=S1} case}
	
	From now on, we focus on the full-Julia case \(J(f)=\Sone\).
	
	In this case, the set of points where the empirical measures converge
	\[
	C_{f}:=\{x\in \Sone:\ \emp_n(x)\text{ converges in }\cP(\Sone)\}
	\]
	will be of major importance, since it is an $f$-invariant set.

	\begin{lemma}\label{lem:C-invariant}
		The set $C_{f}$ is Borel and $f$-invariant.
	\end{lemma}
	
	\begin{proof}
		To see that $C_{f}$ is Borel measurable note that, for each \(n\ge1\), the map
	\[
	x\longmapsto \emp_n(x)\in \cP(\Sone)
	\]
	is continuous. Since \((\cP(\Sone),\Wone)\) is a compact metric space, a sequence
	converges if and only if it is Cauchy. Therefore
	\[
	C_{f}=
	\bigcap_{m\ge1}\ \bigcup_{N\ge1}\ \bigcap_{n,k\ge N}
	\left\{x\in \Sone:\ \Wone(\emp_n(x),\emp_k(x))<\frac1m\right\}.
	\]
	Since, by continuity of \(x\mapsto \emp_n(x)\), each set inside braces is open, we conclude that \(C_{f}\) is Borel.
	
		For invariance, note that
		\[
		\Wone(\emp_n(f(x)),\emp_n(x))
		=
		\Wone\!\left(
		\frac1n\sum_{j=1}^{n}\delta_{f^j(x)},
		\frac1n\sum_{j=0}^{n-1}\delta_{f^j(x)}
		\right)
		\le
		\frac{\diam(\Sone)}{n}\to0.
		\]
		Hence $\emp_n(x)$ converges if and only if $\emp_n(f(x))$ converges.
	\end{proof}
	
	Therefore, Theorem~\ref{thm:mane-E} implies:
	
	\begin{corollary}\label{cor:C-01}
		If $J(f)=\Sone$, then $\Leb(C_{f})\in\{0,1\}$.
	\end{corollary}
    
\subsection{If \texorpdfstring{$C_{f}$}{Cf} has full-measure, we are done}
	
In the case where $\Leb(C_{f})=1$, we can show that emergence is minimal. But to show that, we need a lemma:
	
	\begin{lemma}\label{lem:ae-constancy}
		Assume that every invariant Borel set has Lebesgue measure $0$ or $1$.
		Let $(Y,d)$ be a separable metric space and let $\Phi\colon \Sone\to Y$ be Borel with
		\[
		\Phi\circ f=\Phi.
		\]
		Then $\Phi$ is Lebesgue-a.e.\ constant.
	\end{lemma}
	
	\begin{proof}
		Let
		\[
		\nu:=\Phi_*\Leb
		\]
		be the pushforward of Lebesgue by $\Phi$.
		
		We first show that every open set in $Y$ has $\nu$-measure either $0$ or $1$.
		Indeed, if $U\subset Y$ is open, then
		\[
		A_U:=\Phi^{-1}(U)
		\]
		is Borel in $\Sone$.
		Since $\Phi\circ f=\Phi$, we have
		\[
		x\in A_U
		\iff \Phi(x)\in U
		\iff \Phi(f(x))\in U
		\iff f(x)\in A_U,
		\]
		hence
		\[
		f^{-1}(A_U)=A_U.
		\]
		So $A_U$ is invariant.
		By hypothesis,
		\[
		\nu(U)=\Leb(\Phi^{-1}(U))=\Leb(A_U)\in\{0,1\}.
		\]
		
		We now show that $\nu$ is a Dirac mass.
		
	First, notice that $\supp(\nu)\neq\varnothing$. Indeed, if $\supp(\nu)=\varnothing$, then every point of $Y$ has an open neighborhood of $\nu$-measure $0$.
		Since $Y$ is a separable metric space, we could obtain a countable subcover of $Y$ of balls of measure zero, hence
		\[
		\nu(Y)=0,
		\]
		contradicting $\nu(Y)=1$.
		
	On the other hand, $\supp(\nu)$ contains at most one point. Suppose not, i.e., there exist $y_1\neq y_2$ in $\supp(\nu)$.
		Set
		\[
		r:=\frac13 d(y_1,y_2)>0.
		\]
		Then the open balls
		\[
		B_1:=B(y_1,r),\qquad B_2:=B(y_2,r)
		\]
		are disjoint.
		Because $y_1,y_2\in \supp(\nu)$, both $B_1$ and $B_2$ have positive $\nu$-measure.
		Since every open set has measure $0$ or $1$, both actually have measure $1$,
		contradicting disjointness.
		Hence $\supp(\nu)=\{y_0\}$ for some $y_0\in Y$.
		
	This implies that $\nu=\delta_{y_0}$, since for every $n\ge1$, the open ball
		\[
		B\!\left(y_0,\frac1n\right)
		\]
		meets the support, hence has positive $\nu$-measure.
		So it has measure $1$.
		By continuity from above,
		\[
		\nu(\{y_0\})
		=
		\lim_{n\to\infty}\nu\!\left(B\!\left(y_0,\frac1n\right)\right)
		=
		1.
		\]
		Thus $\nu=\delta_{y_0}$.
		
		Finally,
		\[
		1=\nu(\{y_0\})=\Leb(\Phi^{-1}(\{y_0\})),
		\]
		so $\Phi(x)=y_0$ for Lebesgue-a.e.\ $x$.
	\end{proof}
	
	This lemma implies that the empirical measure function is constant a.e. and hence the emergence of $f$ is minimal:
	
	\begin{corollary}\label{cor:C-full-unique}
		If $\Leb(C_{f})=1$, then there exists an invariant probability measure $\mu$ such that
		\[
		\emp_n(x)\wto \mu
		\qquad\text{for Lebesgue-a.e. }x.
		\]
		In particular,
		\[
		\mathcal{E}_{\Leb}(f)(\eps)=1
		\qquad\forall \eps>0,
		\]
		so $f$ has low emergence with respect to Lebesgue measure.
	\end{corollary}
	
	\begin{proof}
	Choose some $\eta_0\in\cP(\Sone)$ and define
	\[
	\widetilde\Phi(x):=
	\begin{cases}
		\displaystyle\lim_{n\to\infty}\emp_n(x), & x\in C_{f},\\
		\eta_0, & x\notin C_{f}.
	\end{cases}
	\]

	For every $n\ge1$, the map $x\mapsto\emp_n(x)$ is continuous. Hence, on the Borel set $C_{f}$, the map
	\[
	x\longmapsto\lim_{n\to\infty}\emp_n(x)
	\]
	is the pointwise limit of Borel maps and is therefore Borel. Since
	$\widetilde\Phi$ is obtained by extending this map by the constant
	$\eta_0$ on the Borel set $\Sone\setminus C_{f}$, it follows that
	$\widetilde\Phi$ is Borel.

	Moreover, the estimate in the proof of Lemma~\ref{lem:C-invariant} shows that
	\[
	\Wone(\emp_n(f(x)),\emp_n(x))\longrightarrow0.
	\]
	In particular, $x\in C_{f}$ if and only if $f(x)\in C_{f}$, and, whenever
	$x\in C_{f}$, the two sequences have the same limit. Therefore
	\[
	\widetilde\Phi\circ f=\widetilde\Phi
	\qquad\text{on }\Sone.
	\] 
    Lemma~\ref{lem:ae-constancy} now implies that $\widetilde\Phi$ is
	Lebesgue-a.e.\ constant, say equal to some $\mu\in\cP(\Sone)$. Since $\Leb(C_{f})=1$, we obtain $\emp_n(x)\wto\mu$ for Lebesgue-a.e. $x$. By Remark~\ref{rem:asymptotic-invariant}, the measure $\mu$ is $f$-invariant.

	Finally, taking the single measure $\mu$ in the definition of emergence gives $\mathcal E_{\Leb}(f)(\eps)=1$ for all $\eps>0$.
\end{proof}

	\begin{remark}
		So from now on the only remaining case is
		\[
		J(f)=\Sone,\qquad |\deg f|\ge2,\qquad \Leb(C_{f})=0.
		\]
	\end{remark}
	
	\section{Reduction to finitely many indifferent fixed points}

    We shall reduce the remaining case to maps satisfying the following four assumptions. Given a \(C^2\) immersion \(g:\Sone\to\Sone\), we say that \(g\) satisfies (\ref{hyp:H1})--(\ref{hyp:H4}) if:

\begin{enumerate}[label=(H\arabic*),ref=H\arabic*]
	\item\label{hyp:H1}
	\(g\) is topologically conjugate to the linear expanding map
	of degree \(d=\deg(g)\), with \(|d|\ge2\);

	\item\label{hyp:H2}
	there exists a finite set
	\[
	P\subset\operatorname{Fix}(g)
	\]
	such that every \(p\in P\) satisfies
	\[
	|Dg(p)|=1;
	\]
	we call \(P\) the set of indifferent fixed points;

	\item\label{hyp:H3}
	every \(q\in\operatorname{Per}(g)\setminus P\) satisfies
	\[
	|Dg^{\pi_g(q)}(q)|>1,
	\]
	where \(\pi_g(q)\) denotes the least period of \(q\) for \(g\);

	\item\label{hyp:H4}
	for Lebesgue-a.e.\ \(x\in\Sone\),
	\[
	\#\cA_g(x)>1.
	\]
\end{enumerate}

\begin{proposition}\label{prop:reduction-finite-neutral}
Assume
\[
J(f)=\Sone,\qquad |\deg f|\ge2,\qquad \Leb(C_f)=0.
\]
Then there exists \(N\ge1\) such that the iterate
\[
F:=f^N
\]
satisfies assumptions \emph{\ref{hyp:H1}--\ref{hyp:H4}}.

In particular, to prove low emergence for \(f\), it is enough to prove
low emergence for \(F\).
\end{proposition}

\begin{proof}
	Since $J(f)=\Sone$ and $|\deg f|\ge2$, the monotone semiconjugacy to $z\mapsto z^d$, where $d=\deg(f)$, has no plateaux, hence it is a topological conjugacy.

	By Corollary~\ref{cor:mane-bounded-periods}, the set of periods of non-hyperbolic periodic orbits is bounded. Since $f$ is conjugate to $z\mapsto z^d$, and the latter has only finitely many periodic points of each fixed period, it follows that $f$ has only finitely many non-hyperbolic periodic orbits.

	Let $N$ be a common multiple of the periods of these orbits (with $N=1$ if there are none) and set $F:=f^N$. Since $f$ is conjugate to $z\mapsto z^d$, the map $F$ is conjugate to $z\mapsto z^{d^N}$, so assumption~(\ref{hyp:H1}) holds.

	Let \(P\) be the set of non-hyperbolic periodic points of \(f\). As observed above, \(P\) is finite. We first check what happens to these points after passing to the iterate \(F=f^N\).

Fix \(p\in P\), and let \(r=\pi_f(p)\) be its least period for \(f\). Recall that every point of $P$ is fixed for $F$. We will show that it is also indifferent.

By the choice of \(N\), the integer \(r\) divides \(N\). Writing $N=\ell r$, the chain rule and the fact that \(p\) is fixed by \(f^r\), give
\[
DF(p)
=
Df^N(p)
=
D(f^r)^\ell(p)
=
\bigl(Df^r(p)\bigr)^\ell.
\]
Since \(p\) is non-hyperbolic for \(f\),
\[
|Df^r(p)|=1,
\]
and therefore
\[
|DF(p)|
=
|Df^r(p)|^\ell
=
1.
\]
Hence every point of \(P\) is an indifferent fixed point of \(F\).

We now check that passing to the iterate \(F=f^N\) does not create any
new non-hyperbolic periodic points. Suppose that \(q\) is a non-hyperbolic periodic point of \(F\), and let \(k=\pi_F(q)\) be its least period for \(F\). Then $F^k(q)=q$, that is, $f^{Nk}(q)=q$. Thus \(q\) is also periodic for \(f\). Let \(r=\pi_f(q)\) be its least
period for \(f\). Since \(f^{Nk}(q)=q\), the least period \(r\) divides
\(Nk\). Writing $Nk=\ell r$, the chain rule gives
\[
DF^k(q)
=
Df^{Nk}(q)
=
D(f^r)^\ell(q)
=
\bigl(Df^r(q)\bigr)^\ell.
\]
Since \(q\) is non-hyperbolic for \(F\), $|DF^k(q)|=1$ and, consequently,
\[
1
=
|Df^r(q)|^\ell,
\]
and hence
\[
|Df^r(q)|=1.
\]
Therefore \(q\) is already a non-hyperbolic periodic point of the original
map \(f\), so \(q\in P\).

We have thus proved that \(P\) is precisely the set of non-hyperbolic
periodic points of \(F\), and that every point of \(P\) is an indifferent
fixed point of \(F\). In particular, assumption~(\ref{hyp:H2}) holds.

We claim that \(F\) has no sinks. Indeed, \(F\) is topologically
conjugate to \(z\mapsto z^{d^N}\), and the latter has no attracting
periodic orbits. Hence \(F\) has no attracting periodic orbits either.

Now let \(q\in\operatorname{Per}(F)\setminus P\). By the preceding
paragraph, \(P\) is precisely the set of non-hyperbolic periodic points
of \(F\). Therefore $|DF^{\pi_F(q)}(q)|\neq1$. Since \(F\) has no sinks, we cannot have $|DF^{\pi_F(q)}(q)|<1$. Consequently,
\[
|DF^{\pi_F(q)}(q)|>1,
\]
and (\ref{hyp:H3}) holds.

	Finally, by Corollary~\ref{cor:iterate-convergence}, $C_F\subset C_f$. Since $\Leb(C_f)=0$, it follows that $\Leb(C_F)=0$. On the other hand, $\cP(\Sone)$ is compact, so a sequence of probability measures with a unique accumulation point must converge. Consequently, for every $x\notin C_F$,
	\[
	\#\cA_F(x)>1.
	\]
	Hence, (\ref{hyp:H4}) holds for Lebesgue almost every $x\in\Sone$.

	Therefore, $F$ satisfies (\ref{hyp:H1})--(\ref{hyp:H4}). By
	Corollary~\ref{cor:iterate-emergence}, low emergence for $F$ implies low emergence for $f$, which concludes the proof.
\end{proof}

\section{Maps with finitely many indifferent fixed points}
Throughout this section, we consider a \(C^2\) immersion $f:\Sone\to\Sone$ satisfying assumptions (\ref{hyp:H1})--(\ref{hyp:H4}).

The goal of this section is to prove the following. 

\begin{proposition}\label{prop.lowemergenceunderh1h4}
Let \(f:\Sone\to\Sone\) be a \(C^2\) immersion satisfying assumptions
\emph{(\ref{hyp:H1})--(\ref{hyp:H4})}. Then \(f\) has low emergence.
\end{proposition}

Notice that this proposition completes the proof of Theorem~\ref{main}.


\subsection{An estimate on Lyapunov exponents}

Given $\mu\in\cps$, for ease of notation, we shall denote its Lyapunov exponent by  
\[
L(\mu)\eqdef\int \log|Df(x)|d\mu(x).
\]
In this subsection, we shall prove:

\begin{lemma}
\label{lem:expoentepositivo}
Let $f$ be a $C^2$ immersion of $\mathbb{S}^1$ satisfying assumptions~\emph{(\ref{hyp:H1})--(\ref{hyp:H3})}. Given an ergodic $f$-invariant measure $\mu\in\cP_f^{\operatorname{erg}}(\mathbb{S}^1)$, if $\mu\neq\delta_p$ for every $p\in P$ then $L(\mu)>0$. 
\end{lemma}

\subsubsection{The fully periodic case}    

Notice that by assumption~(\ref{hyp:H3}) we may assume that $\Lambda\eqdef\operatorname{supp}(\mu)$ is not a periodic orbit. Consider the set 
\[
\Lambda^{\infty}\eqdef\bigcap_{n\geq 0}f^n(\Lambda). 
\]
\begin{lemma}
    \label{lem:igualdadeinfinito}
$\Lambda^\infty=\Lambda$    
\end{lemma}
\begin{proof}
It suffices to notice that $f(\Lambda)=\Lambda$. Indeed, take $x\in\Lambda$ and $V$ an open neighborhood around  $f(x)$. Then, $f^{-1}(V)$ is an open neighbourhood around $x$ and by invariance
\[
\mu(V)=\mu(f^{-1}(V))>0.
\]
Thus, $f(x)\in\Lambda$ and therefore $f(\Lambda)\subset\Lambda$. For the reverse inclusion, since clearly one has $\Lambda\subset f^{-1}(f(\Lambda))$ one obtains
\[
1=\mu(\Lambda)\leq\mu(f^{-1}(f(\Lambda)))=\mu(f(\Lambda)).
\]
Therefore, $f(\Lambda)$ is a compact set with full measure, which implies that $\Lambda\subset f(\Lambda)$. This finishes the proof that $f(\Lambda)=\Lambda$ and that $\Lambda^\infty=\Lambda$. 
\end{proof}

\begin{lemma}
    \label{lem:expoenteperiodicopositivo}
Let $f$ be a $C^2$ immersion of $\mathbb{S}^1$ satisfying assumptions~\emph{(\ref{hyp:H1})--(\ref{hyp:H3})}. Given an ergodic $f$-invariant measure $\mu\in\cP_f^{\operatorname{erg}}(\mathbb{S}^1)$, if $\mu\neq\delta_p$ for every $p\in P$ then $\mu\left(\Lambda^{\infty}\cap\per\right)>0$ implies $L(\mu)>0$.
\end{lemma}
\begin{proof}
Assume that $\mu\left(\Lambda^{\infty}\cap\per\right)>0$. By assumption~(\ref{hyp:H1}), $\per$ is a countably infinite set. Thus, there must exist some $q\in\per$ so that $\mu(\{q\})>0$. By invariance and ergodicity, it follows that $\mu(O(q))=1$. Assumption~(\ref{hyp:H3}) implies then $L(\mu)>0$ as desired.
\end{proof}

\subsubsection{Adapted intervals and return maps}

Let us recall now some definitions and results from \cite{mane} which will play a major role in our argument. We say that a sequence $\tix=(x_{-n})_{n\in\N}\subset \Lambda$ is a \emph{coherent sequence of $x\in\Lambda$} if $x_0=x$ and $f(x_{-n})=x_{-n+1}$, i.e. $\tix$ is a choice of past for $x$ inside $\Lambda$. We denote by $\Sigma(x,\Lambda)$ the set of all coherent sequences of $x\in\Lambda$. Given $J\subset\mathbb{S}^1$, a map $\varphi: J\to\mathbb{S}^1$ is said to be \emph{an inverse branch of $f^n$ inside $J$} if $f^n\circ\varphi=Id_{J}$. A \emph{coherent sequence of inverse branches} is a sequence $\{\varphi_n\}_{n\in\N}$ of inverse branches of $f^n$ inside $J$ such that $f\circ\varphi_{n+1}=\varphi_n$.

\begin{definition}
We say that an open interval $J\subset\mathbb{S}^1$ is \emph{adapted} to $\Lambda$ if $J\cap\Lambda\neq\emptyset$ and for every $x\in J\cap\Lambda$, for every coherent sequence $\tix\in\Sigma(x,\Lambda)$ there exists a coherent sequence of inverse branches $\{\varphi_n\}_{n\in\N}$ satisfying, for all $n>0$,
\begin{enumerate}[label=(\Roman*)]
    \item $\varphi_n(x)=x_{-n}$;
    \item either $\varphi_n(J)\subset J$ or $\varphi_n(J)\cap J=\emptyset$
\end{enumerate}
\end{definition}

The following corresponds to Lemma I.2 of \cite{mane} applied directly in our case. 

\begin{lemma}
    \label{lem:maneI2}
Let $f$ be a $C^2$ immersion of $\mathbb{S}^1$ satisfying assumptions~\emph{(\ref{hyp:H1})}. Then, for every $x\in\Lambda^\infty\setminus\per$ there exists an open interval $J\supset\{x\}$ adapted to $\Lambda$.    
\end{lemma}

\begin{definition}
We say that a map
$\psi:J\to J$ is a \emph{return map} of $\Lambda$ if there exists
$m\ge 1$ such that $\psi$ is a branch of $f^{-m}|_J$ satisfying
$f^j(\psi(J))\cap J=\emptyset$ for all $0<j<m$ and there exists
$x\in J\cap\Lambda^\infty$ such that
$\psi(x)\in\Lambda^\infty$. We denote $R(\Lambda,J)$ the set of return maps $\psi:J\to J$ of $\Lambda$.
\end{definition}
 The next result is Lemma~I.3 of \cite{mane}, with the correction
given in \cite{mane-erratum}, specialized to our setting of a
$C^2$ circle immersion.

\begin{lemma}
 \label{lem:maneI3} 
 If $f:\T\to\T$ is a $C^2$ immersion satisfying assumption~\emph{(\ref{hyp:H1})} and
$\Lambda\subset\T$ is a compact invariant set then, if
$\Lambda^\infty$ contains non-periodic points, there
exists an interval $J$ adapted to $\Lambda$ and constants
$K_1>0$, $0<\lambda<1$ such that every coherent sequence
$(J,\{\varphi_n\})$ associated to $\Lambda$ satisfies
\begin{align}
\sum_{m=1}^{\infty}\left|\varphi_m'(x)\right|
    &\le K_1, \tag{2}\\
\left|\varphi_n'(x)\right|
    &\le K_1\left|\varphi_n'(y)\right|, \tag{3}
\end{align}
for all $x\in J$, $y\in J$, $n\ge 1$, and
\begin{equation}
|\psi'(x)|\le\lambda
\tag{4}
\end{equation}
for all $x\in J$, $\psi\in R(\Lambda,J)$.
\end{lemma}

\subsubsection{Estimating the Poincaré map}

We return now to the particular setting in which $\Lambda=\operatorname{supp}(\mu)$, where $\mu$ is an ergodic measure with $\mu(P)=0$. By Lemma~\ref{lem:expoenteperiodicopositivo} we may assume that $\mu$-almost every point in $\Lambda^\infty$ is non-periodic. Then, Lemma~\ref{lem:maneI2} ensures the existence of an open interval $J\supset\{x\}$ adapted to $\Lambda$. In particular, we have $\mu(J)>0$, and we may define the almost everywhere finite return time function for $y\in J$
\[
\tau(y)=\min\{n>0;f^n(y)\in J\}.
\]
Let $\Phi(y)\eqdef f^{\tau(y)}(y)$ be the almost everywhere well-defined Poincaré first return map. Observe that, by continuity, $\tau$ is locally constant and thus at each point where $\Phi$ is well defined it is also differentiable. We shall establish the following.

\begin{lemma}
    \label{lem:estimapoincare}
Let $\lambda\in(0,1)$ be the constant given in Lemma~\ref{lem:maneI3}. Then, for $\mu$-almost every $y\in J$ it holds $|D\Phi(y)|\geq \lambda^{-1}$.
\end{lemma}
\begin{proof}
We fix $y\in J\cap\Lambda^{\infty}$ a non-periodic point which returns to $J$ infinitely often. Fix $m=\tau(y)$ and set $x=\Phi(y)=f^m(y)$. By Lemma~\ref{lem:igualdadeinfinito} we have $x\in\Lambda^\infty$. Then, there exists a coherent sequence $\tix=(x_{-n})_{n\in\N}\in\Sigma(x,\Lambda)$. Since $J$ is an adapted interval there exists a coherent sequence of inverse branches $(J,\{\varphi_n\})$ such that $\varphi_n(x)=x_{-n}$. In particular, we have $\varphi_m(x)=y$. We claim that $\varphi_m\in R(\Lambda,J)$. 

Indeed, since $J$ is an adapted interval for each $0<j<m$ we either have $f^j(\varphi_m(J))\subset J$ or $f^j(\varphi_m(J))\cap J=\emptyset$. If the latter possibility occurs, then we would have, in particular, 
\[
f^j(\varphi_m(x))\in J,
\]
for some $0<j<m$, which is absurd since $\varphi_m(x)=y$ and $m=\tau(y)$. 

Applying (4) of Lemma~\ref{lem:maneI3} we thus get 
\[
|D\varphi_m(x)|\leq\lambda.
\]
Since $\varphi_m$ is an inverse branch of $f^{-m}$ this implies that 
\[
|Df^m(y)|\geq\lambda^{-1},
\]
as claimed. This completes the proof of the lemma.
\end{proof}

\subsubsection{Proof of Lemma~\ref{lem:expoentepositivo}}

In view of Lemma~\ref{lem:estimapoincare}, Lemma~\ref{lem:expoentepositivo} follows directly from the abstract relation between Lyapunov exponents of the Poincaré first return map with the exponent of $f$, made precise in the lemma below.

\begin{lemma}
\label{lem:induzexpoente}
Let $\nu\eqdef\mu|_J/\mu(J)$ be the normalized restriction of $\mu$ to the adapted interval $J$. Then,
\[
\int_J\log|D\Phi(y)|d\nu(y)=\frac{1}{\mu(J)}\int_{\mathbb{S}^1}\log|Df(y)|d\mu(y).
\]
\end{lemma}
\begin{proof}
First, notice that $\nu$ is a $\Phi$-invariant measure and the system $(\Phi,\nu)$ is ergodic. Then, denote 
$\tau_n(y)=\sum_{\ell=0}^{n-1}\tau(\Phi^{\ell}(y))$ the Birkhoff sum of the first return time function. By the chain rule, we have
\[
D\Phi^\ell(y)=Df^{\tau_\ell(y)}(y).
\]
Thus, 
\begin{equation}
    \label{eq:phicomf}
    \sum_{\ell=0}^{n-1}\log|D\Phi(\Phi^\ell(y))|=\sum_{i=0}^{\tau_n(y)-1}\log|Df(f^i(y))|.
\end{equation}
Now, divide by $n$ on both sides of \eqref{eq:phicomf}. By ergodicity of the system $(\Phi,\nu)$, the left hand side of \eqref{eq:phicomf} converges to $\int_J\log|D\Phi(y)|d\nu(y)$. Applying Kac's lemma and using ergodicity of $(f,\mu)$ we get that the right-hand side of \eqref{eq:phicomf} converges to $1/\mu(J)\int_{\mathbb{S}^1}\log|Df(y)|d\mu(y)$, concluding the proof.
\end{proof}

Combining Lemma~\ref{lem:estimapoincare} with Lemma~\ref{lem:induzexpoente} we thus get 
\[
L(\mu)=\int_{\mathbb{S}^1}\log|Df(y)|d\mu(y)=\mu(J)\int_J\log|D\Phi(y)|\geq\mu(J)\log(\lambda^{-1})>0.
\]
\qed

\begin{corollary}[Classification of ergodic components of typical asymptotic measures]
There exists a full Lebesgue measure set $\Gamma\subset\T$ such that for every $x\in\Gamma$ and every $\mu\in\cA_f(x)$, for almost every ergodic component $m$ of $\mu$, exactly one of the following holds:
\begin{enumerate}
\item $m$ is a Dirac measure supported on an indifferent periodic point;
\item $m$ has positive Lyapunov exponent and is a physical measure.
\end{enumerate}
\end{corollary}
\begin{proof}
Let $\Gamma$ be the set given by Theorem~\ref{thm:entropyfor} and take $x\in\Gamma$. Then, every $\mu\in\cA_f(x)$ satisfies the entropy formula. By Lemma~\ref{lem:decomposicao}, almost every ergodic component $m$ of $\mu$ also satisfies the entropy formula. 
If $m$ is not supported on an indifferent periodic orbit, then, by Lemma~\ref{lem:expoentepositivo}, we have $L(m)>0$. By Lemma~\ref{lem:quian}, it follows that $m$ is a physical measure, concluding the proof. 
\end{proof}

\subsection{Low emergence}

The results we have proved so far enable us to write the following

\begin{corollary}
Let $f:\T\to\T$ be a $C^2$ immersion satisfying \emph{(\ref{hyp:H1})--(\ref{hyp:H4})}. Then, there exists a set $\Gamma\subset\T$ with $\Leb(\Gamma)=1$ such that if $x\in\Gamma$ then $\cA_{f}(x)\subset\overline{\operatorname{conv}}(\{\delta_p;p\in P\})$. In particular, $f$ has low emergence.
\end{corollary}

\begin{proof}
Assumption~(\ref{hyp:H4}) prohibits $f$ from possessing physical measures. Therefore, the results of the preceding subsection taken together imply the existence of a full measure set $\Gamma\subset\T$ so that for every $x\in\Gamma$ and every measure $\mu\in\cA_{f}(x)$, every ergodic component $\mu_{\omega}$ of $\mu$ is a Dirac mass supported in some point in $P$. The low emergence claim follows from Proposition~\ref{prop:compact-family-bound}.
\end{proof}


\bibliographystyle{plain}
\bibliography{Biblio}

\end{document}